\documentclass[11pt]{article}
\usepackage{amsmath,amsthm}
\usepackage{amssymb}
\usepackage{enumerate}
\usepackage[mathscr]{eucal}
\usepackage[cm]{fullpage}
\usepackage[english]{babel} 
\usepackage[latin1]{inputenc}
\newcommand{\T}{\mathscr{T}}
\newcommand{\On}{\mathscr{O}_n}
\newcommand{\End}{\mathrm{End}}
\newcommand{\wEnd}{\mathrm{wEnd}}
\newcommand{\sEnd}{\mathrm{sEnd}}
\newcommand{\swEnd}{\mathrm{swEnd}}
\newcommand{\Aut}{\mathrm{Aut}}
\theoremstyle{plain}
\newtheorem{theorem}{Theorem}[section]
\newtheorem{proposition}[theorem]{Proposition}
\newtheorem{lemma}[theorem]{Lemma}
\newtheorem{corollary}[theorem]{Corollary}
\def\im{\mathop{\mathrm{Im}}\nolimits} 
\def\Ker{\mathop{\mathrm{Ker}}\nolimits} 
\def\rank{\mathop{\mathrm{rank}}\nolimits} 
\def\N{\mathbb N}

\begin{document}

\title{On the Rank of Monoids of Endomorphisms of a Finite Directed Path\footnote{
This work is funded by national funds through the FCT - Funda\c c\~ao para a Ci\^encia e a Tecnologia, I.P., under the scope of the project UIDB/00297/2020 (Center for Mathematics and Applications).}}

\author{V.H. Fernandes and T. Paulista}

\maketitle

\renewcommand{\thefootnote}{}



\renewcommand{\thefootnote}{\arabic{footnote}}
\setcounter{footnote}{0}

\begin{abstract} 
In this paper we consider endomorphisms of a finite directed path from monoid generators perspective. Our main aim is to determine the rank of the monoid 
$\wEnd\vec{P}_n$ of all weak endomorphisms of a directed path with $n$ vertices, which is 
a submonoid of the widely studied monoid $\On$ of all order-preserving transformations of a $n$-chain. 
Also, we describe the regular elements of $\wEnd\vec{P}_n$ and calculate its size and number of idempotents. 
\end{abstract}

\section*{Introduction and Preliminaries} 

Endomorphisms of graphs allow to establish natural connections between Graph Theory and Semigroup Theory 
in the same way that automorphisms of graphs allow it between Graph Theory and Group Theory.  
This fact has led many authors to have studied during the last decades various algebraic and combinatorial properties of monoids of endomorphisms of graphs. In particular, regularity, in the sense of Semigroup Theory, is one of the most studied properties. A general solution to the problem, posed in 1987 by Knauer and Wilkeit, see \cite{Marki:1988}, of which graphs have a regular monoid of endomorphisms seems to be very difficult to obtain. Nevertheless, for some special classes of graphs, various authors studied and solved this question (for instance, see 
\cite{
Fan:1993,
Fan:1997,
Hou&Gu:2016,
Hou&Gu&Shang:2014,
Hou&Luo&Fan:2012,
Hou&Song&Gu:2017,
Knauer&Wanichsombat:2014,
Li:2006,
Li&Chen:2001,
Wilkeit:1996}). 

\smallskip 

A very important invariant of a semigroup (monoid or group), which has been object of intensive research in Semigroup Theory, 
is the notion of \textit{rank}, i.e. the least number of generators of a semigroup (monoid or group) $S$, which we denote by $\rank(S)$.

\smallskip 

Let $\Omega$ be a finite set
with at least $3$ elements.
It is well-known that the symmetric group  of $\Omega$
has rank $2$ (as a semigroup, a monoid or a group) and the monoid of all (full) transformations
and the monoid of all partial transformations of $\Omega$ have
ranks $3$ and $4$, respectively.
The survey \cite{Fernandes:2002} presents 
these results and similar ones for other classes of transformation monoids,
in particular, for monoids of order-preserving transformations and
for some of their extensions. 
For example, the rank of the extensively studied monoid of all order-preserving transformations of a $n$-chain is $n$. This result was proved by Gomes and Howie \cite{Gomes&Howie:1992} in 1992. 
More recently, for instance, the papers 
\cite{
Araujo&al:2015,
Fernandes&al:2014,
Fernandes&al:2018ip,
Fernandes&Quinteiro:2014,
Fernandes&Sanwong:2014,
Zhao&Fernandes:2015} 
are dedicated to the computation of the ranks of certain classes of transformation semigroups or monoids.

\smallskip 

For a finite set $\Omega$, denote by $\T(\Omega)$ the monoid of all (full) transformations of $\Omega$. 

Let $G=(V,E)$ be an undirected [respectively, a directed] graph, without loops and without multiple edges. 
Let $\alpha$ be an element of $\T(V)$. We say that the transformation $\alpha$ is: 
\begin{itemize}
\item an \textit{endomorphism} of $G$ if $\{u,v\}\in E$ implies  $\{u\alpha,v\alpha\}\in E$
[respectively, $(u,v)\in E$ implies  $(u\alpha,v\alpha)\in E$], for all $u,v\in V$;
\item a \textit{weak endomorphism} of $G$ if $\{u,v\}\in E$ and $u\alpha\ne v\alpha$ imply  $\{u\alpha,v\alpha\}\in E$ 
[respectively, if $(u,v)\in E$ and $u\alpha\ne v\alpha$ imply  $(u\alpha,v\alpha)\in E$], for all $u,v\in V$;
\item a \textit{strong endomorphism} of $G$ if $\{u,v\}\in E$ if and only if  $\{u\alpha,v\alpha\}\in E$ 
[respectively, $(u,v)\in E$ if and only if  $(u\alpha,v\alpha)\in E$], for all $u,v\in V$;
\item a \textit{strong weak endomorphism} of $G$ if $\{u,v\}\in E$ and $u\alpha\ne v\alpha$ if and only if $\{u\alpha,v\alpha\}\in E$ 
[respectively, $(u,v)\in E$ and $u\alpha\ne v\alpha$ if and only if $(u\alpha,v\alpha)\in E$], for all $u,v\in V$;
\item an \textit{automorphism} of $G$ if $\alpha$ is a bijective strong endomorphism (i.e. $\alpha$ is bijective and $\alpha$ and $\alpha^{-1}$ are both endomorphisms). For finite graphs (undirected or directed), any bijective endomorphism is an automorphism. 
\end{itemize}

We denote by:
\begin{itemize}
\item $\End G$ the set of all endomorphisms of $G$;
\item $\wEnd G$ the set of all weak endomorphisms of $G$;
\item $\sEnd G$ the set of all strong endomorphisms of $G$;
\item $\swEnd G$ the set of all strong weak endomorphisms of $G$;
\item $\Aut G$ the set of all automorphisms of $G$. 
\end{itemize}

It is clear that $\End G$, $\wEnd G$, $\sEnd G$, $\swEnd G$ and $\Aut G$ are monoids under composition of maps, 
i.e. they are submonoids of $\T(V)$. Moreover, $\Aut G$ is also a group.  Clearly, 
$$
\Aut G\subseteq\sEnd G\subseteq\End G, \swEnd G \subseteq\wEnd G
$$ 
\begin{center}
\begin{picture}(55,140)(0,10)
\put(40,60){\line(-1,1){40}}
\put(40,60){\line(1,1){40}}
\put(80,100){\line(-1,1){40}}
\put(0,100){\line(1,1){40}}
\put(40,20){\line(0,1){40}}
\put(37,16.5){$\bullet$}\put(45,16.5){$\Aut G$}
\put(37,137.5){$\bullet$}\put(45,140){$\wEnd G$}
\put(77,98){$\bullet$}\put(84,100){$\End G$}
\put(-2,98){$\bullet$}\put(-45,100){$\swEnd G$}
\put(37.5,58){$\bullet$}\put(45,55){$\sEnd G$}
\end{picture}
\end{center}
(these inclusions may not be strict).

\medskip 

Let $\N$ be the set of all natural numbers greater than zero and let $n \in \N$. 

Let $P_n$ be an undirected path with $n$ vertices, e.g. 
$
P_n=\left(\{1,\ldots,n\},\{\{i,i+1\}\mid i=1,\ldots,n-1\}\right). 
$
The number of endomorphisms of $P_n$ has been determined by Arworn 
\cite{Arworn:2009} (see also the paper \cite{Michels&Knauer:2009} by Michels and Knauer). 
In addition, several other combinatorial and algebraic properties of $P_n$ were also studied in these two papers and, 
for instance, in \cite{Arworn&Knauer&Leeratanavalee:2008,Hou&Luo&Cheng:2008}. 
More recently, in \cite{Dimitrova&Fernandes&Koppitz&Quinteiro:2020}, the first author et al., for $n\geqslant 2$, showed that: 
\begin{itemize}
\item $\rank(\wEnd P_n)=n+\sum_{j=1}^{\lfloor\frac{n-3}{3}\rfloor}\lfloor\frac{n-3j-1}{2}\rfloor$;
\item $\rank(\swEnd P_n)=\lceil\frac{n}{2}\rceil+1$; 
\item $\rank(\End P_n)=1+\lfloor\frac{n-1}{2}\rfloor+\sum_{j=1}^{\lfloor\frac{n-3}{3}\rfloor}\lfloor\frac{n-3j-1}{2}\rfloor$; 
\item $\rank(\sEnd P_n)=1$ for $n\ne 3$, $\rank(\sEnd P_3)=3$; 
\item $\rank(\Aut P_n)=1$. 
\end{itemize}

\medskip 

In this paper, we consider the directed path with $n$ vertices 
$$
\vec{P}_n=\left(\{1,\ldots,n\},\{(i,i+1)\mid i=1,\ldots,n-1\}\right) 
$$
and the associated monoids $\Aut\vec{P}_n$, $\sEnd\vec{P}_n$, $\End\vec{P}_n$, $\swEnd\vec{P}_n$ and $\wEnd\vec{P}_n$. 
In Section \ref{basics}, we show that $\Aut\vec{P}_n$, $\sEnd\vec{P}_n$ and $\End\vec{P}_n$ are trivial monoids and 
that $\swEnd\vec{P}_n$ is simply made up of constant maps and the identity. 
Still in this section, we present some basic properties of $\wEnd\vec{P}_n$, 
calculate its cardinal and number of idempotents and describe its regular elements. 
Our main result states that the rank of the monoid $\wEnd\vec{P}_n$ is $n-1$ and is proved in Section \ref{rank}. 

\medskip 

For general background on Semigroup
Theory and standard notation, we refer the reader to Howie's book \cite{Howie:1995}. 
On the other hand, regarding Algebraic Graph Theory, our main reference is Knauer's book \cite{Knauer:2011}.  

\section{Basic Properties} \label{basics} 

Let $n\in\N$. Let $\T_n=\T(\{1,\ldots ,n\})$. 

\smallskip 

 Recall that a transformation $\alpha$ of $\{1,\ldots,n\}$ is said to be \textit{order-preserving} 
if $x\leqslant y$ implies $x\alpha\leqslant y\alpha$, for all $x,y\in\{1,\ldots ,n\}$. 
Denote by $\On$ the submonoid of $\T_n$ of all order-preserving transformations.   

\smallskip 

First, we show that $\wEnd\vec{P_n}$ is a submonoid of $\On$ and so 
that will also be the case of all other monoids of different type of endomorphisms of $\vec{P_n}$. 

\begin{proposition} \label{order-preserving}
One has $\wEnd\vec P_n\subseteq\On$. 
\end{proposition}

\begin{proof} 
Let $\alpha \in\wEnd\vec{P}_{n}$. 

Let $i,j \in \{1,\ldots ,n\}$ be such that $i \leqslant j$. Then $j=i+k$, for some $k \geqslant 0$.

If $k=0$ then $i=j$ and, obviously, $i \alpha \leqslant j \alpha$.

If $k=1$ then $(i, j)$ is an edge of $\vec{P}_{n}$. Obviously, if $i \alpha = j \alpha$ then $i \alpha \leqslant j \alpha$. So suppose that $i \alpha \neq j \alpha$. Since $\alpha \in\wEnd\vec{P}_{n}$, we have that $(i \alpha, j \alpha)$ is an edge of $\vec{P}_{n}$, whence $j \alpha = i \alpha + 1$ and so $i \alpha \leqslant j \alpha$.

Now, suppose $k>1$. By the previous case, we have 
$i \alpha \leqslant (i+1) \alpha, (i+1) \alpha \leqslant (i+2) \alpha,\ldots, (i+k-1) \alpha \leqslant (i+k) \alpha = j \alpha$. 
Hence, $i \alpha \leqslant j \alpha$, as required.
\end{proof}  

Let $\alpha\in\T_n$. We denote by $\im(\alpha)$ the \textit{image} of $\alpha$, i.e. $\im(\alpha)=\{x\alpha\mid x\in \{1,\ldots,n\}\}$. 

\begin{lemma} \label{interval}
Let $\alpha \in\wEnd\vec{P}_{n}$. Then $\im(\alpha)$ is an interval of $\{1,\ldots ,n\}$.
\end{lemma}
\begin{proof}
Let $\alpha \in\wEnd\vec{P}_{n}$.

If $|\im(\alpha)|=1$  then $\im(\alpha)$ is an interval of $\{1,\ldots ,n\}$.

Suppose that $|\im(\alpha)|>1$ and, in order to reach a contradiction, that $\im(\alpha)$ is not an interval of $\{1,\ldots ,n\}$. 
Therefore, there exists $a \in\im(\alpha)$ such that $a+1,\ldots ,a+t-1 \notin\im(\alpha)$ and $a+t \in\im(\alpha)$, 
for some $t \geqslant 2$.

Let $i^{*}= \max \big \{i \in \{1,\ldots ,n\}\mid i \alpha = a \big \}$ and let $j\in \{1,\ldots ,n\}$ be such that $j\alpha=a+t$. 
Since $\alpha\in\On$ and $i^*\alpha=a<a+t=j\alpha$, we have $i^*<j$, whence $i^*<n$ and so $(i^*,i^*+1)$ is an edge of $\vec{P}_n$. 
On the other hand, by definition of $i^*$, we have $(i^*+1)\alpha\neq a =i^*\alpha$. 
It follows that $(i^*\alpha,(i^*+1)\alpha)$ is also an edge of $\vec{P}_n$ and so $a+1=i^*\alpha+1=(i^*+1)\alpha\in\im(\alpha)$, 
which is a contradiction. 

Thus, $\im(\alpha)$ must be an interval of $\{1,\ldots ,n\}$, as required.
\end{proof} 

Let us denote by $1_n$ the identity map of $\{1,\ldots ,n\}$.

\begin{theorem}
One has  $\End\vec{P}_{n}=\{1_{n}\}$. 
\end{theorem} 
\begin{proof} 
Obviously, $1_{n} \in$ End$\vec{P}_{n}$.

Let $\alpha \in$ End$\vec{P}_{n}$.

If $n=1$ then, trivially, $\alpha = 1_{n}$.

Let $n \geqslant 2$. For all $i \in \{1,\ldots ,n-1\}$, $(i,i+1)$ is an edge of $\vec{P}_{n}$. 
Therefore, for all $i \in \{1,\ldots ,n-1\}$, $(i \alpha, (i+1) \alpha)$ is also an edge of $\vec{P}_{n}$ (since $\alpha \in$ End$\vec{P}_{n}$), 
which implies that $(i+1) \alpha = i \alpha +1$, for all $i \in \{1,\ldots ,n-1\}$. 
Hence $2 \alpha = 1 \alpha +1, 3 \alpha = 2\alpha+1=1 \alpha +2,\ldots ,n \alpha= (n-1)\alpha+1=1 \alpha +n-1$. 
As $1\alpha\geqslant 1$ and $n\alpha\leqslant n$, we must have $1 \alpha =1$ and $n\alpha=n$. 
It follows that $2 \alpha =2, 3 \alpha =3,\ldots ,(n-1) \alpha =n-1$, whence $\alpha = 1_{n}$, as required. 
\end{proof} 

Since $\Aut\vec{P}_{n} \subseteq\sEnd\vec{P}_{n} \subseteq\End\vec{P}_{n}=\{1_{n}\}$, we also get 
$$
\Aut\vec{P}_{n} = \sEnd\vec{P}_{n} =\End\vec{P}_{n}=\{1_{n}\}.
$$
For the remaining submonoid $\swEnd\vec{P}_{n}$ of $\wEnd\vec{P}_{n}$, we have:  

\begin{theorem}
One has $\swEnd\vec{P}_{n}=\left \{
\begin{pmatrix} 
1 & \cdots & n \\ 
i & \cdots & i \\
\end{pmatrix}
\mid 1 \leqslant i \leqslant n \right \} \cup \{1_{n}\}$. 
\end{theorem} 
\begin{proof} 
Clearly, $1_{n} \in\swEnd\vec{P}_{n}$ and $\begin{pmatrix} 
1 & \cdots & n \\ 
i & \cdots & i \\
\end{pmatrix} \in\swEnd\vec{P}_{n}$, for all $i \in \{1,\ldots ,n\}$.

Let $\alpha \in\swEnd\vec{P}_{n}$.

If $n=1$ then, trivially, $\alpha = \begin{pmatrix} 1 \\ 1 \\ \end{pmatrix} = 1_{n}$.

Let $n \geqslant 2$. Then, for all $i \in \{1,\ldots ,n-1\}$, since $(i,i+1)$ is and edge of $\vec{P}_{n}$, 
we have $i \alpha = (i+1) \alpha$ or $(i \alpha, (i+1) \alpha)$ is an edge of $\vec{P}_{n}$ (i.e. $(i+1) \alpha = i \alpha +1$).

Suppose that $\alpha\neq 1_n$. Then, as $\alpha\in \On$, there exists $i\in\{1,\ldots,n-1\}$ such that $i\alpha=(i+1)\alpha$. 

Take $i^{*}= \min \big \{i \in \{1,\ldots ,n-1\}\mid i \alpha = (i+1)\alpha\big \}$. 
If $i^*\geqslant 2$ then $(i^*-1)\alpha\neq i^*\alpha$ and so $(i^*-1)\alpha =i^*\alpha-1=(i^*+1)\alpha-1$, 
whence $((i^*-1)\alpha,(i^*+1)\alpha)$  
is an edge of $\vec{P}_{n}$. It follows that $(i^*-1,i^*+1)$ must be an edge of $\vec{P}_{n}$, which is a contradiction. 
Thus $i^*=1$. 

Let $k^{*}= \max \big \{k \in \{1,\ldots ,n-1\}\mid (1+k)\alpha = 1\alpha\big \}$. 
Notice that $1\alpha=2\alpha=\cdots=(1+k^*)\alpha$. 
If $k^*<n-1$ then $(2+k^*)\alpha\neq (1+k^*)\alpha$ and so $(2+k^*)\alpha= (1+k^*)\alpha+1=k^*\alpha+1$,
whence  $(k^*\alpha,(2+k^*)\alpha)$ is an edge of $\vec{P}_{n}$. It follows that $(k^*, 2+k^*)$ must be an edge of $\vec{P}_{n}$, 
which is again a contradiction. Thus $k^*=n-1$. 

Therefore, $\alpha$ is a constant map, as required. 
\end{proof} 

Next, we give a characterization of the weak endomorphisms of $\vec{P}_{n}$. 

\begin{proposition}\label{char}
Let $\alpha\in\T_n$. Then, $\alpha \in\wEnd\vec{P}_{n}$ if and only if $\alpha \in\On$ and $\im(\alpha)$ is an interval of $\{1,\ldots ,n\}$.
\end{proposition}
\begin{proof}
The direct implication is an immediate consequence of Proposition \ref{order-preserving} and Lemma \ref{interval}.

Let us prove the converse implication. Let $\alpha \in \On $ and suppose that $\im(\alpha)$ is an interval of $\{1,\ldots ,n\}$. 
Let $i \in \{1,\ldots ,n-1\}$ be such that $i \alpha \neq (i+1) \alpha$. Since $\alpha \in \On $, it follows that $i \alpha < (i+1) \alpha$. 
Then $i \alpha +1 \leqslant (i+1) \alpha$.  
Since $\im(\alpha)$ is an interval of $\{1,\ldots ,n\}$, 
$i\alpha\leqslant i \alpha +1 \leqslant (i+1)\alpha$ and $i \alpha, (i+1) \alpha \in\im(\alpha)$, we deduce that $i \alpha +1 \in\im(\alpha)$. 
Then, there exists $j \in \{1,\ldots ,n\}$ such that $j \alpha = i \alpha +1$. 
Hence, we have $\alpha \in \On$ and $i\alpha<j\alpha$, from which follows that $i<j$, whence $i+1\leqslant j$ and so 
$(i+1)\alpha\leqslant  j\alpha=i\alpha+1\leqslant (i+1)\alpha$.  
Therefore, we obtain $(i+1) \alpha = i \alpha +1$, i.e. $(i \alpha, (i+1)\alpha)$ is an edge of $\vec{P}_{n}$, as required.
\end{proof} 

\begin{theorem}\label{card} 
One has $|\wEnd\vec{P}_{n}|=\displaystyle \sum_{k=1}^{n} (n-k+1) \begin{pmatrix} n-1 \\ k-1 \\ \end{pmatrix}$. 
\end{theorem} 
\begin{proof}  
Let us begin by observing that $\wEnd\vec{P}_{n}$ can be written in the following manner:
$$
\wEnd\vec{P}_{n}=\displaystyle\bigcup_{k=1}^{n} J_{k},
$$
where $J_{k}=\{\alpha \in\wEnd\vec{P}_{n}\mid |\im(\alpha)|=k\}$, for all $k \in \{1,\ldots ,n\}$. 
Clearly, the sets $J_1,\ldots,J_{k}$ are disjoint, whence 
$$
|\wEnd\vec{P}_{n}|=\displaystyle\sum_{k=1}^{n}|J_{k}|.
$$

Let $k\in\{1,\ldots ,n\}$ and $\alpha \in J_{k}$. 
Then $\im(\alpha)$ is an interval of $\{1,\ldots ,n\}$ with $k$ elements and so $\im(\alpha)=\{j+1,\ldots ,j+k\}$, for some $j \in \{0,\ldots ,n-k\}$. 
Hence, the number of possibilities for $\im(\alpha)$ is $n-k+1$.

Since $\alpha \in\wEnd\vec{P}_{n}$, it follows that $(i+1)\alpha=i\alpha$ or $(i+1)\alpha=i\alpha+1$, for all $i \in \{1,\ldots ,n\}$. 
Then $\alpha$ has the form
$$
\alpha =
\left(\begin{array}{ccc}1&\cdots&i_{1}\\j+1&\cdots&j+1\\\end{array}\bigg|\begin{array}{ccc}i_{1}+1&\cdots&i_{1}+i_{2}\\j+2&\cdots&j+2\\\end{array}\bigg|\cdots\bigg|\begin{array}{ccc}i_{1}+\cdots+i_{k-1}+1&\cdots&i_{1}+\cdots+i_{k}\\j+k&\cdots&j+k\\\end{array}\right),
$$
for some $i_{1},\ldots ,i_{k}\in\{1,\ldots ,n\}$ such that $i_{1}+\cdots+i_{k}=n$. 
Observe that $i_{t}=\big|\big\{i\in\{1,\ldots ,n\}:i\alpha=j+t\big\}\big|$, for all $t\in\{1,\ldots ,k\}$. Therefore, 
we can associate the above transformation with the sequence $(i_{1},i_{2},\ldots ,i_{k})$, 
where $i_{1},\ldots ,i_{k}\in\{1,\ldots ,n\}$ and $i_{1}+\cdots+i_{k}=n$, i.e. with an \textit{ordered partition} of $n$ with $k$ elements.

It is known that the number of ordered partitions of $n$ with $k$ elements is given by 
$\begin{pmatrix} n-1 \\ k-1 \\ \end{pmatrix}$. 
Consequently, for each one of the possibilities for $\im(\alpha)$ there are $\begin{pmatrix} n-1 \\ k-1 \\ \end{pmatrix}$ 
possibilities for the transformation $\alpha$. Hence 
$$
\displaystyle |J_{k}|=(n-k+1)\begin{pmatrix} n-1 \\ k-1 \\ \end{pmatrix}
$$
and, therefore, 
$$
\displaystyle |\wEnd\vec{P}_{n}|= \sum_{k=1}^{n}(n-k+1)\begin{pmatrix} n-1 \\ k-1 \\ \end{pmatrix},
$$
as required. 
\end{proof} 

\medskip 

Recall that an element $s$ of a semigroup $S$ is called \textit{regular} if there exists $x \in S$ such that $s = sxs$. 
Moreover, a semigroup is said to be \textit{regular} if all its elements are regular.

\smallskip 

Now, we give a description of the regular elements of $\wEnd\vec{P}_{n}$. 

\begin{proposition}\label{regular} 
Let $\alpha\in\wEnd\vec{P}_{n}$. Then, the following properties are equivalent:
\begin{enumerate}
\item $\alpha$ is a regular element of $\wEnd\vec{P}_{n}$;
\item $|x\alpha^{-1}|>1$ implies $x\alpha^{-1}\cap\{1,n\}\neq\emptyset$, for all $x\in\im(\alpha)$;
\item $|x\alpha^{-1}|>1$ implies $x\in\{\min\im(\alpha),\max\im(\alpha)\}$, for all $x\in\im(\alpha)$. 
\end{enumerate}
\end{proposition}
\begin{proof}
First, observe that, clearly, properties 2 and 3 are equivalent. 

We begin by supposing that $\alpha$ is a regular element of $\wEnd\vec{P}_{n}$. 
Take $\beta\in\wEnd\vec{P}_{n}$ such that $\alpha=\alpha\beta\alpha$. 
Let  $x\in\im(\alpha)$ be such that $|x\alpha^{-1}|>1$. 
Let $i\in\{1,\ldots,n\}$ and $k\geqslant1$ be such that $x\alpha^{-1}=\{i,i+1,\ldots,i+k\}$. 
By contradiction, admit that $x\not\in\{\min\im(\alpha),\max\im(\alpha)\}$. 
Then $1< i<i+k<n$, $(i-1)\alpha=x-1$ and $(i+k+1)\alpha=x+1$.  
If $x\beta=(x-1)\beta$ then 
$$
x=i\alpha=i\alpha\beta\alpha=x\beta\alpha=(x-1)\beta\alpha=(i-1)\alpha\beta\alpha=(i-1)\alpha=x-1,
$$ 
a contradiction. 
If $x\beta=(x+1)\beta$ then 
$$
x=i\alpha=i\alpha\beta\alpha=x\beta\alpha=(x+1)\beta\alpha=(i+k+1)\alpha\beta\alpha=(i+k+1)\alpha=x+1,
$$ 
a contradiction. 
Hence, we have $(x-1)\beta=x\beta-1$ and $(x+1)\beta=x\beta+1$. Thus,  
$$
(x\beta-1)\alpha=(x-1)\beta\alpha=(i-1)\alpha\beta\alpha=(i-1)\alpha=x-1,
$$ 
$$
x\beta\alpha=i\alpha\beta\alpha=i\alpha=x
$$ 
and 
$$
(x\beta+1)\alpha=(x+1)\beta\alpha=(i+k+1)\alpha\beta\alpha=(i+k+1)\alpha=x+1,
$$ 
whence 
$x\beta-1,x\beta+1\not\in x\alpha^{-1}$ and $x\beta\in x\alpha^{-1}$. It follows that $x\beta-1<i$ and $x\beta+1>i+k\geqslant i+1$, 
which implies $x\beta\leqslant i$ and $x\beta>i$, a contradiction. Therefore, we conclude that $x\in\{\min\im(\alpha),\max\im(\alpha)\}$. 

Conversely, suppose that $\alpha$ satisfies property 2 (or 3). Then $\alpha$ has the form
$$
\alpha=
\begin{pmatrix} 
1 & \cdots & i-1 & i & i+1& \cdots & i+k & i+k+1 & \cdots & n \\ 
j & \cdots & j & j & j+1 & \cdots &j+k&j+k & \cdots & j+k\\
\end{pmatrix},
$$
for some $k\geqslant0$ and $1\leqslant i,j\leqslant n-k$. Define 
$$
\beta=
\begin{pmatrix} 
1 & \cdots & j-1 & j & j+1& \cdots & j+k & j+k+1 & \cdots & n \\ 
i & \cdots & i & i & i+1 & \cdots &i+k&i+k & \cdots & i+k\\
\end{pmatrix}. 
$$
Then, clearly, $\beta\in\wEnd\vec{P}_{n}$ and $\alpha\beta\alpha=\alpha$, 
which shows that $\alpha$ is a regular element of $\wEnd\vec{P}_{n}$, 
as required. 
\end{proof}

Observe that, it is easy to check that, for $n\leqslant3$, all elements of $\wEnd\vec{P}_{n}$ verify property 2 (or 3) above. 
On the other hand, for $n\geqslant4$, 
$$
\begin{pmatrix} 
1 & 2 & 3 & 4 & \cdots & n \\ 
1 & 2 & 2 & 3 & \cdots & n-1\\
\end{pmatrix} 
$$
is an element of  $\wEnd\vec{P}_{n}$ which does not satisfy property 2 (or 3) above. 
Therefore, we have:

\begin{corollary}
The monoid $\wEnd\vec{P}_{n}$ is regular if and only if $n\leqslant3$. 
\end{corollary}

Recall that an element $\alpha\in\T_n$ is idempotent (i.e. $\alpha^2=\alpha$) if and only if it fixes all the elements of its image 
(i.e. $x\alpha=x$, for all $x\in\im(\alpha)$). Since an idempotent is a regular element, 
we deduce from Proposition \ref{regular} that an idempotent of $\wEnd\vec{P}_{n}$ has the form
$$
\alpha=
\begin{pmatrix} 
1 & \cdots & i-1 & i & i+1& \cdots & i+k & i+k+1 & \cdots & n \\ 
i & \cdots & i & i & i+1 & \cdots &i+k&i+k & \cdots & i+k\\
\end{pmatrix},
$$
for some $k\geqslant0$ and $1\leqslant i\leqslant n-k$, 
and so the number of idempotents of $\wEnd\vec{P}_{n}$ coincides with the number of non-empty intervals of $\{1,\ldots,n\}$, i.e. 
$\sum_{i=1}^{n}(n-i+1)=\frac{1}{2}(n+n^2)$. 

\begin{corollary}
The monoid $\wEnd\vec{P}_{n}$ has $\frac{1}{2}(n+n^2)$ idempotents. 
\end{corollary}

The table below gives us an idea of the size of $\wEnd\vec P_n$ and of the number $\text{e}_n$ of idempotents in this monoid. 
\begin{center}
$\begin{array}{|c|c|c|}\cline{1-3}
n & |\wEnd\vec P_n|  & \text{e}_n\\ \cline{1-3}  
1 & 1& 1\\ \cline{1-3}
2 & 3 &  3 \\ \cline{1-3}
3 & 8 &  6 \\ \cline{1-3}
4 &  20     & 10  \\ \cline{1-3}
5  & 48      &  15 \\ \cline{1-3}
6  & 112    &  21 \\ \cline{1-3}
7  &  256   &  28 \\ \cline{1-3}
8  &   576  &  36 \\ \cline{1-3}
\end{array}$ \quad
$\begin{array}{|c|c|c|}\cline{1-3}
n &  |\wEnd\vec P_n| & \text{e}_n\\ \cline{1-3}
9  &  1280    &  45 \\ \cline{1-3}
10  &  2816   &  55 \\ \cline{1-3}
11 & 6144   &   66\\ \cline{1-3}
12 &  13312    &  78 \\ \cline{1-3}
13 &   28672   &  91 \\ \cline{1-3}
14 &   61440  &   105\\ \cline{1-3}
15 &   131072    &   120\\ \cline{1-3}
16 &   278528   &  136 \\ \cline{1-3}
\end{array}$
\end{center}

\section{The Rank of $\wEnd\vec P_n$} \label{rank} 

In this section we compute the rank of $\wEnd\vec P_n$.   
We proceed by determining a generating set of minimal size. 

\smallskip 

First, we observe that $\wEnd\vec P_1=\left\{\begin{pmatrix} 1 \\ 1 \\ \end{pmatrix}\right\}$ and 
$\wEnd\vec P_2=\left\{\begin{pmatrix} 1 &2 \\ 1 &1 \\ \end{pmatrix},\begin{pmatrix} 1 &2 \\ 1 &2 \\ \end{pmatrix},\begin{pmatrix} 1 &2 \\ 2 &2 \\ \end{pmatrix}\right\}$ and so, clearly, $\wEnd\vec P_1$ has rank $0$ and $\wEnd\vec P_2$ has rank $2$. 

\smallskip 

Let us consider $n\geqslant 3$ and define 
$$
\alpha_{i}=
\begin{pmatrix} 
1 & \cdots & i & i+1 & \cdots & n \\ 
1 & \cdots & i & i & \cdots & n-1\\
\end{pmatrix}
\quad\text{and}\quad
\beta_{i}=
\begin{pmatrix} 
1 & \cdots & i & i+1 & \cdots & n \\ 
2 & \cdots & i+1 & i+1 & \cdots & n\\
\end{pmatrix},
$$
for $i=1,\ldots ,n-1$. It is clear that $\alpha_{1},\ldots ,\alpha_{n-1},\beta_{1},\ldots ,\beta_{n-1}\in \wEnd\vec P_n$. 

\begin{lemma} \label{generators 1}
Let $n\geqslant 3$. Then  $\{\alpha \in\wEnd\vec{P}_{n}\mid |\im(\alpha)| = n-1\} \subseteq \langle \alpha_{1},\ldots ,\alpha_{n-2},\beta_{n-1} \rangle$. 
\end{lemma} 
\begin{proof} 
Let us begin by proving that $\{\alpha \in\wEnd\vec{P}_{n}\mid|\im(\alpha)|= n-1\} = \{ \alpha_{1},\ldots ,\alpha_{n-1},\beta_{1},\ldots ,\beta_{n-1}\}$.

Let $\alpha \in\wEnd\vec{P}_{n}$ be such that $|\im(\alpha)|=n-1$. By Proposition \ref{char}, $\im(\alpha)$ is an interval of $\{1,\ldots ,n\}$, whence 
$\im(\alpha)=\{1,..,n-1\}$ or $\im(\alpha)=\{2,..,n\}$.
Since $(i+1) \alpha = i \alpha$ or $(i \alpha, (i+1)\alpha)$ is an edge of $\vec{P}_{n}$ (i.e. $(i+1) \alpha = i \alpha +1$), for all $i \in \{1,\ldots ,n-1\}$, and 
$|\im(\alpha)|=n-1$, there exists a unique $j \in \{1,\ldots ,n-1\}$ such that $(j+1) \alpha = j \alpha$ and, 
for all $i \in \{1,\ldots ,n-1\}\setminus \{j\}$, $(i+1) \alpha = i \alpha +1$. Consequently, $\alpha = \alpha_{j}$, if $\im(\alpha)=\{1,..,n-1\}$, 
and $\alpha = \beta_{j}$, if $\im(\alpha)=\{2,..,n\}$.

Now, in order to prove the desired inclusion, it is enough to observe that $\alpha_{n-1}=\beta_{n-1}\alpha_{1}$ and $\beta_{i}=\alpha_{i}\beta_{n-1}$, 
for all $i \in \{1,\ldots ,n-2\}$. 
\end{proof}

\begin{lemma} \label{generators 2}
Let $n\geqslant 3$ and let $\alpha \in\wEnd\vec{P}_{n}$ be such that $|\im(\alpha)|=k$, for some $1 \leqslant k \leqslant n-2$. 
Then, there exist $\gamma_{1},\gamma_{2} \in\wEnd\vec{P}_{n}$ such that $\alpha=\gamma_{1}\gamma_{2}$ and $|\im(\gamma_{i})|=k+1$, $i =1,2$.
\end{lemma} 
\begin{proof} 
Let $\alpha \in\wEnd\vec{P}_{n}$ be such that $|\im(\alpha)|=k$, for some $k \in \{1,\ldots ,n-2\}$.

Since $\im(\alpha)$ is an interval of $\{1,\ldots ,n\}$ we have $\im(\alpha)=\{j+1,\ldots ,j+k\}$, for some $j \in \{0,\ldots ,n-k\}$. On the other hand, 
since $\alpha \in \On$, the transformation $\alpha$ has the form
$$
\alpha =
\left(\begin{array}{ccc}1&\cdots&i_{1}\\j+1&\cdots&j+1\\\end{array}\bigg|\begin{array}{ccc}i_{1}+1&\cdots&i_{1}+i_{2}\\j+2&\cdots&j+2\\\end{array}\bigg|\cdots\bigg|\begin{array}{ccc}i_{1}+\cdots+i_{k-1}+1&\cdots&i_{1}+\cdots+i_{k}\\j+k&\cdots&j+k\\\end{array}\right),
$$
for some $i_{1},\ldots ,i_{k}\in\{1,\ldots ,n\}$ such that $i_{1}+\cdots+i_{k}=n$. 

\smallskip

We will consider several cases. 

\smallskip 

Case 1: $k=1$. 

If $j=n-k=n-1$ then
$$
\alpha=\left(\begin{array}{ccc}1&\cdots&n\\n&\cdots&n\\\end{array}\right).
$$
By taking  
$$
\gamma_{1}=\left(\begin{array}{cccc}1&\cdots&n-1&n\\2&\cdots&2&3\\\end{array}\right)
\quad\text{and}\quad 
\gamma_{2}=\left(\begin{array}{cccc}1&2&\cdots&n\\n-1&n&\cdots&n\\\end{array}\right),
$$
clearly, we have $\gamma_{1},\gamma_{2} \in\wEnd\vec{P}_{n}$, $\alpha=\gamma_{1}\gamma_{2}$ 
and $|\im(\gamma_{1})|=|\im(\gamma_{2})|=k+1=2$.

If $0 \leqslant j \leqslant n-k-1=n-2$ then
$$
\alpha=\left(\begin{array}{ccc}1&\cdots&n\\j+1&\cdots&j+1\\\end{array}\right).
$$
In this case, let us define 
$$
\gamma_{1}=\left(\begin{array}{cccc}1&\cdots&n-1&n\\1&\cdots&1&2\\\end{array}\right)
\quad\text{and}\quad 
\gamma_{2}=\left(\begin{array}{ccccc}1&2&3&\cdots&n\\j+1&j+1&j+2&\cdots&j+2\\\end{array}\right).
$$
So, it is also clear that $\gamma_{1},\gamma_{2} \in\wEnd\vec{P}_{n}$, 
$\alpha=\gamma_{1}\gamma_{2}$ and $|\im(\gamma_{1})|=|\im(\gamma_{2})|=k+1=2$.

\smallskip 

Case 2: $2 \leqslant k \leqslant n-2$.

Let $p=$ max $ \big\{t\in \{1,\ldots ,k\}\mid i_{t} \geqslant 2 \big\}$.

If $j=n-k$ then 
$$
\begin{array}{lll}
\alpha =
\left(\begin{array}{ccc}1&\cdots&i_{1}\\n-k+1&\cdots&n-k+1\\\end{array}\bigg|\cdots\bigg|\begin{array}{ccc}i_{1}+\cdots+i_{p-1}+1&\cdots&i_{1}+\cdots+i_{p}\\n-k+p&\cdots&n-k+p\\\end{array}\right|\cdots\\\\
\textrm{\hspace{8.8 cm}}\cdots\left|\begin{array}{ccc}i_{1}+\cdots+i_{k-1}+1&\cdots&n\\n&\cdots&n\\\end{array}\right).\end{array}
$$
Take
$$
\begin{array}{lll}
\gamma_{1}=\left(\begin{array}{ccc}1&\cdots&i_{1}\\2&\cdots&2\\\end{array}\bigg|\begin{array}{ccc}i_{1}+1&\cdots&i_{1}+i_{2}\\3&\cdots&3\\\end{array}\bigg|\cdots\bigg|\begin{array}{ccc}i_{1}+\cdots+i_{p-1}+1&\cdots&i_{1}+\dots+i_{p}-1\\p+1&\cdots&p+1\\\end{array}\right|\\\\\textrm{\hspace{6.5 cm}}\left|\begin{array}{c}i_{1}+\cdots+i_{p}\\p+2\\\end{array}\bigg|\cdots\bigg|\begin{array}{ccc}i_{1}+\cdots+i_{k-1}+1&\cdots&n\\k+2&\cdots&k+2\\\end{array}\right)\end{array}
$$
and
$$
\gamma_{2}=\left(\begin{array}{ccccccccc}1&2&\cdots&p+1&p+2&\cdots&k+2&\cdots&n\\n-k&n-k+1&\cdots&n-k+p&n-k+p&\cdots&n&\cdots&n\\\end{array}\right).
$$
Then, it is easy to check that $\gamma_{1},\gamma_{2} \in\wEnd\vec{P}_{n}$, $\alpha=\gamma_{1}\gamma_{2}$ 
and $|\im(\gamma_{1})|=|\im(\gamma_{2})|=k+1$. 

Finally, if $0 \leqslant j \leqslant n-k-1$ then 
$$
\begin{array}{lll}
\alpha =
\left(\begin{array}{ccc}1&\cdots&i_{1}\\j+1&\cdots&j+1\\\end{array}\bigg|\cdots\bigg|\begin{array}{ccc}i_{1}+\cdots+i_{p-1}+1&\cdots&i_{1}+\cdots+i_{p}\\j+p&\cdots&j+p\\\end{array}\right|\cdots\\\\
\textrm{\hspace{6.5 cm}}\cdots\left|\begin{array}{ccc}i_{1}+\cdots+i_{k-1}+1&\cdots&n\\j+k&\cdots&j+k\\\end{array}\right).\end{array}
$$
and by taking 
$$
\begin{array}{lll}
\gamma_{1}=\left(\begin{array}{ccc}1&\cdots&i_{1}\\1&\cdots&1\\\end{array}\bigg|\begin{array}{ccc}i_{1}+1&\cdots&i_{1}+i_{2}\\2&\cdots&2\\\end{array}\bigg|\cdots\bigg|\begin{array}{ccc}i_{1}+\cdots+i_{p-1}+1&\cdots&i_{1}+\cdots+i_{p}-1\\p&\cdots&p\\\end{array}\right|\\\\\textrm{\hspace{6.5 cm}}\left|\begin{array}{c}i_{1}+\cdots+i_{p}\\p+1\\\end{array}\bigg|\cdots\bigg|\begin{array}{ccc}i_{1}+\cdots+i_{k-1}+1&\cdots&n\\k+1&\cdots&k+1\\\end{array}\right)\end{array}
$$
and
$$
\gamma_{2}=\left(\begin{array}{ccccccccc}1&\cdots&p&p+1&\cdots&k+1&k+2&\cdots&n\\j+1&\cdots&j+p&j+p&\cdots&j+k&j+k+1&\cdots&j+k+1\\\end{array}\right), 
$$
we obtain $\gamma_{1},\gamma_{2} \in\wEnd\vec{P}_{n}$, $\alpha=\gamma_{1}\gamma_{2}$ and 
$|\im(\gamma_{1})|=|\im(\gamma_{2})|=k+1$, as required. 
\end{proof}

Now, we deduce that $\{ \alpha_{1},\ldots ,\alpha_{n-2},\beta_{n-1} \}$ is a generating set of $\wEnd\vec{P}_{n}$. 

\begin{proposition} 
Let $n\geqslant 3$. Then $\wEnd\vec{P}_{n} = \langle \alpha_{1},\ldots ,\alpha_{n-2},\beta_{n-1} \rangle$.
\end{proposition} 
\begin{proof} 
In Lemma \ref{generators 1} we have seen that all transformations of $\wEnd\vec{P}_{n}$ of rank $n-1$ can be written as the product of elements of the set $\{\alpha_{1},\ldots ,\alpha_{n-2},\beta_{n-1}\}$. On the other hand, Lemma \ref{generators 2} says that all elements of $\wEnd\vec{P}_{n}$ of rank between $1$ and $n-2$ can be written as the product of elements of higher rank. This means that all elements of $\wEnd\vec{P}_{n}$ of rank between $1$ and $n-1$ can be written as the product of elements of the set $\{\alpha_{1},\ldots ,\alpha_{n-2},\beta_{n-1}\}$. Then $\{\alpha_{1},\ldots ,\alpha_{n-2},\beta_{n-1}\}$ is a generating set of $\wEnd\vec{P}_{n}$.
\end{proof} 

In addition, we will show below that $\{ \alpha_{1},\ldots ,\alpha_{n-2},\beta_{n-1} \}$ is a generating set of $\wEnd\vec{P}_{n}$ of minimal size. 

\smallskip

Let $\alpha\in\T_n$. We denote by $\rank\alpha$ the \textit{rank} of $\alpha$ and by $\Ker(\alpha)$ the \textit{kernel} of $\alpha$, i.e. 
$\rank\alpha=|\im(\alpha)|$ and $\Ker(\alpha)=\{(x,y)\in\{1,\ldots,n\}\times\{1,\ldots,n\}\mid x\alpha=y\alpha\}$. 

\begin{lemma} 
Let $M$ be a submonoid of $\T_n$ such that $1_n$ is the unique element of rank $n$. Let $X$ be a set of generators of $M$. 
Then $X$ has at least one element with each of the possible kernels of elements of rank $n-1$. 
\end{lemma} 
\begin{proof}
Let $\alpha \in M$ be such that rank $\alpha =n-1$.

Since $X$ is a set of generators of $S$, there exist $t \in \mathbb{N}$ and $\gamma_{1},\ldots , \gamma_{t} \in X$ such that $\alpha=\gamma_{1} \cdots \gamma_{t}$ (we may assume that $\gamma_{1},\ldots ,\gamma_{t} \neq 1_n$).

Clearly, $\Ker(\gamma_{1}) \subseteq\Ker(\alpha)$. 
On the other hand, it is easy to check that $\rank(\gamma_{1} \cdots \gamma_{t} )\leqslant\rank\gamma_{i}$, for all $i \in \{1,\ldots ,t\}$ and so, in particular, 
$\rank \alpha \leqslant\rank\gamma_{1}$. Consequently, we have that rank $\gamma_{1}=n-1$ (since $\rank \alpha = n-1$, $\gamma_{1} \neq 1_n$ and $1_n$ is the unique element of $M$ of rank $n$).

Now, from $\Ker(\gamma_{1}) \subseteq\Ker(\alpha)$ and $\rank\gamma_{1}=\rank\alpha$, it follows that $\Ker(\gamma_{1}) =\Ker(\alpha)$. 
Therefore, the set of generators $X$ has at least one element with each of the possible kernels of elements of rank $n-1$, as required. 
\end{proof} 

Finally, since the  transformations $\alpha_{1},\ldots ,\alpha_{n-2},\beta_{n-1}$ have all, clearly, distinct kernels, 
we conclude that $\{ \alpha_{1},\ldots ,\alpha_{n-2},\beta_{n-1} \}$ is a generating set of $\wEnd\vec{P}_{n}$ of minimal size. 
Thus, we may deduce our main result: 

\begin{theorem} 
For $n\geqslant 3$, the rank of the monoid $\wEnd\vec{P}_{n}$ is $n-1$.
\end{theorem}


{\small \sf  
\noindent{\sc V\'\i tor H. Fernandes}, 
CMA, Departamento de Matem\'atica, 
Faculdade de Ci\^encias e Tecnologia, 
Universidade NOVA de Lisboa, 
Monte da Caparica, 
2829-516 Caparica, 
Portugal; 
e-mail: vhf@fct.unl.pt. 

\medskip 

\noindent{\sc T\^ania Paulista}, 
Departamento de Matem\'atica, 
Faculdade de Ci\^encias e Tecnologia, 
Universidade NOVA de Lisboa, 
Monte da Caparica, 
2829-516 Caparica, 
Portugal; 
e-mail: t.paulista@campus.fct.unl.pt. 
} 

\end{document}